\documentclass[11pt,a4paper]{amsart}
\usepackage[utf8]{inputenc}
\usepackage[T1]{fontenc}
\usepackage{amssymb,amsthm,amsmath}
\usepackage{tikz}
\usepackage{multicol}
\usepackage{wrapfig}

\usetikzlibrary{shapes.geometric, arrows}
\usepackage{url}
\usepackage{hyperref}
\usepackage[a4paper]{geometry}
\newtheorem{theorem}{Theorem}[section]
\newtheorem{lemma}[theorem]{Lemma}
\newtheorem{proposition}[theorem]{Proposition}
\newtheorem{cor}[theorem]{Corollary}  

\newtheorem{definition}[theorem]{Definition}

\theoremstyle{contraction}
\newtheorem{rem}[theorem]{Remark}

\usepackage{pdfsync}

\begin{document}

\author{Philippe  Biane}
\thanks{This research was supported by the project CARPLO, ANR-20-CE40-0007.}
\email  {philippe.biane@univ-eiffel.fr }
\address{Institut Gaspard Monge
UMR CNRS - 8049
Universit\'e Gustave Eiffel
5 boulevard Descartes, 
77454 Champs-Sur-Marne
FRANCE}

\title{Combinatorics of descent algebras and graph coverings}

\begin{abstract}
We give a direct combinatorial proof that the product of two descent classes in a  symmetric  group is a sum of descent classes. The proof is based on the fact that the group product gives a covering map when descent classes are endowed with the graph structure coming from the weak order. The main geometric argument is valid for any Coxeter group, even infinite ones for which the descent algebra does not exist.
\end{abstract}
\maketitle
\section{Introduction}
Solomon \cite{S} has proved that the product of two descent classes, in the group algebra of a finite Coxeter group, is a sum of descent classes, therefore the linear combinations of descent classes form a subalgebra of the group algebra. The proof  however does not yield a direct combinatorial explanation. In this paper we give a simple  argument, involving covering graphs, which shows that the product of two descent classes is a sum of descent classes.  As an outcome we  give  a refinement of the structure coefficients of the descent algebra taking into account the structure of the covering graphs.  For the convenience of readers who are not versed in Coxeter groups we give the proofs for the symmetric groups first and then show how our arguments extend easily to the case of general Coxeter groups. In fact we will see that the geometric argument that we give is also valid  in infinite Coxeter groups, even when the descent algebra is not well defined.

This paper is organized as follows: in section 2 we recall the definition of descents and recoils of permutations.
In section 3 we give our main combinatorial result, Proposition \ref{recoil}, which is at the core of the proof of the existence of the descent algebra. In section 4 we recall some standard notions of graph theory concerning coverings of graphs. The existence of the descent algebra follows then easily, as shown in section 5. In section 6 we give some explicit examples of our construction and in section 7 we explore the action of the fundamental group induced by the covering. 
 Finally in section 8 we show that our main result extends in a simple way to   Coxeter groups, by using the exchange property. The existence of the descent algebra follows then by exactly the same reasoning as for the symmetric groups.  
\section{Descents and recoils}
\label{sec:3}

For a positive integer $n$ let $S_n$ be the symmetric group acting on $\{1,2,\ldots,n\}$.
A permutation $\sigma\in S_n$ with  linear representation  $\sigma_1\sigma_2\ldots\sigma_n$ has a {\sl descent} at $i$ if $\sigma_i>\sigma_{i+1}$.
It has a {\sl recoil} at $i$ if $\sigma^{-1}$ has a descent at $i$, this happens when $i+1$ appears before $i$ in  $\sigma_1\sigma_2\ldots\sigma_n$.
Let $D(\sigma)\subset \{1,2,\ldots,n-1\}$ (resp. $R(\sigma)=D(\sigma^{-1})$) denote the set of descents (resp. recoils) of a permutation.
In the following we will consider recoils rather than descents for clarity of exposition but   it is straightforward to convert our results  into statements about descents. 
For $I\subset \{1,\ldots,n-1\}$ denote by ${\mathcal Y}_I$ the recoil class of $I$, which is the set of all permutations $\sigma$ such that $R(\sigma)=I$ and let
$$Y_I=\sum_{\sigma\in {\mathcal Y}_I}\sigma$$
the sum being in the group algebra of $S_n$ (which we consider over the rationals but any field would do).
Solomon \cite{S} has proved that there exists nonnegative integers $a_{IJK}$, indexed by triples of  subsets of $\{1,\ldots,n-1\}$, such that
\begin{equation}\label{descentalg}Y_IY_J=\sum_{K}a_{IJK}Y_K.\end {equation}
It follows that the linear space generated by the $Y_I$ for $I\subset\{1,2,\ldots,n-1\} $ forms a subalgebra of ${\bf Q}S_n$ called the recoil algebra of the symmetric group.
The proof is indirect, it shows that the sums $$X_I=\sum_{J\subset I}Y_J$$ satisfy a similar formula for some coefficients $b$:
 $$X_IX_J=\sum_{K}b_{IJK}X_K$$
 then (\ref{descentalg}) follows since the $Y$ can be obtained from the $X$ by M\"obius inversion
 $$Y_I=\sum_{J\subset I}(-1)^{|I|-|J|}X_I.$$
 Since the sets ${\mathcal Y}_I$ are disjoint the coefficients $a_{IJK}$ are nonnegative integers.
 The combinatorial significance of the existence of the algebra however remains obscure.
 In this paper I want to explain 
 some of the combinatorics behind (\ref{descentalg}) using a simple property of products of permutations, Proposition \ref{recoil} below, and some notions of  graph theory.
\section{Descent and recoil classes}We denote by $s_i$ the simple transposition $(i\, i+1)$ so that $S=\{s_1,s_2,\ldots,s_{n-1}\}$ is  a generating subset  of $S_n$. The associated right Cayley graph  has $S_n$  for vertex set and an edge between $\pi$ and $\sigma$ if $\sigma=\pi s_i$ for some $i$. Since the $s_i$ are involutions the graph is non-oriented. The associated length function $\ell(\sigma)$ is the smallest number of factors in a decomposition $\sigma=s_{i_1}\ldots s_{i_k}$, and is also the number of inversions of the permutation $\sigma$. The distance function on the graph is $d(\pi,\sigma)=\ell(\pi^{-1}\sigma)$, it allows to define the  
right weak order $\leq_w$ on $S_n$:
one has $\pi\leq_w\sigma$ if $d(e,\pi)+d(\pi,\sigma)=d(e,\sigma)$ (here $e$ is the identity permutation). The permutation  $w_0(i)=n+1-i$ is the maximal element for the right weak order. The graph can be realized as the skeleton of a polyhedron of dimension $n-1$, the permutahedron.

 Note that $\sigma$ has a recoil at $i$ if and only if $\ell(s_i\sigma )=\ell(\sigma)-1$. 
We will endow the sets ${\mathcal Y}_I$ with the graph structure induced from the Cayley graph of $S_n$.
The following is well-known.
\begin{proposition}\label{recoilclass}
Let $I\subset \{1,\ldots,n-1\}$, 
the set ${\mathcal Y}_I$ has a minimal and a maximal element for the right weak order, called 
$\alpha_I$ and $\beta_I$, moreover
$${\mathcal Y}_I=[\alpha_I,\beta_I]=\{\sigma\in S_n\,|\,\alpha_I\leq_w\sigma\leq_w\beta_I\}.$$

\end{proposition}
 The group generated by the $s_i$,  for $i\in I$, is a product of symmetric groups acting on intervals of the form $[k,l]$ and $\alpha_I$ is the product of the maximal elements (for the weak order) of these symmetric groups, for example if $n=12$ and $I=\{1,4,5,6,9,10\}$, then the product is $S_{[1,2]}\times S_{[3]}\times S_{[4,7]}\times S_{[8]}\times S_{[9,11]}\times S_{[12]}$ and 
$\alpha_I=21|3|7654|8|\,11\, 10\, 9|\,12$, where we put bars between the successive descending runs. In order to find $\beta_I$ note that $\sigma\to \sigma w_0$ is an order reversing
 bijection between $R(I)$ and $R(I^{c})$ therefore $\beta_I=\alpha_{I^c}w_0$. In our example $I^c=\{2,3,7,8,11\}$ which gives
 $\beta _I=11\,12|\,10|\,789|6|5|234|1$, where we put bars between ascending runs.
 In particular ${\mathcal Y}_I$ is a connected subset of  the Cayley graph.
A criterion for two neighbouring permutations in the Cayley graph of $S_n$ to belong to the same recoil class is the following.
\begin{lemma}\label{>2}
Let $\sigma\in S_n$ and  $i\in\{1,\ldots,n-1\}$ then 
$R(\sigma)=R(\sigma s_i)$
 if and only if
$|\sigma_i-\sigma_{i+1}|\geq 2$.
\end{lemma}
\begin{proof} 
If $\sigma=\sigma_1\sigma_2\ldots\sigma_i\sigma_{i+1}\ldots\sigma_n$ then $\sigma s_i=\sigma_1\sigma_2\ldots\sigma_{i+1}\sigma_i\ldots\sigma_n$. 
A recoil is created or destroyed by permuting 
$\sigma_i$ and $\sigma_{i+1}$ if and only if 
$|\sigma_i-\sigma_{i+1}|=1$.
\end{proof}
 The following result is the main combinatorial property of recoil classes that we will use.

\begin{proposition}\label{recoil}
Let $\sigma,\pi,\rho\in S_n$ and $i\leq n-1$ be such that $$\sigma=\pi\rho \quad 
\text{and} \quad R(\sigma s_i)=R(\sigma)$$  then there exists a unique factorization 
$\sigma s_i=\pi'\rho'$ where either : 
\begin{enumerate}
\item
$\pi=\pi'$ while $\rho^{-1}\rho'\in S$ and $R(\rho)=R(\rho')$ or
\item
$\rho=\rho'$ while $\pi^{-1}\pi'\in S$ and $R(\pi)=R(\pi')$ .
\end{enumerate}
\end{proposition}
\begin{proof}
There exist  two factorizations $\sigma s_i=\pi'\rho'$ such that either $\pi=\pi'$ or $\rho=\rho'$, namely

\begin{itemize}
\item
$\pi'=\pi$, $\rho'=\rho s_i$
\item
 $\pi'=\pi(\rho s_i\rho^{-1})$, $\rho'=\rho$
 \end{itemize}

Suppose that $R(\rho)=R(\rho s_i)$, by Lemma \ref{>2} this happens if and only if $|\rho_i-\rho_{i+1}|\geq 2$, then the factorization
$\sigma s_i=\pi\rho'=\pi(\rho s_i)$ satisfies (1). In the other factorization, $\sigma s_i=\pi'\rho$, one has $\pi'=\pi t$ where $t$ is the transposition $(\rho_i\,\rho_{i+1})$. Since $t$ is not a simple transposition one has $\pi^{-1}\pi'\notin S$ therefore  case $(1)$ of the Proposition holds but not $(2)$.

Suppose now that $R(\rho)\ne R(\rho s_i)$,
this happens if and only if $|\rho_i-\rho_{i+1}|=1$. In this case let
$j=\min (\rho_i,\rho_{i+1})$ and $j+1=\max (\rho_i,\rho_{i+1})$ then
$\rho s_i\rho^{-1}=s_j$  and $|\pi_j-\pi_{j+1}|=|\sigma_i-\sigma_{i+1}|\geq 2$ therefore, by Lemma \ref{>2},
$R(\pi)=R(\pi s_j)$. Now case $(2)$ holds but not $(1)$.

Finally we see that  either $(1)$ or $(2)$ holds and  Proposition \ref{recoil} is proved.
\end{proof}
\section{Covering of graphs}
We consider graphs  which are simple, non-oriented and without loops.

\begin{definition}\label{covering} Let $G$ and $H$ be graphs with respective vertex sets $V,\,W$ and edge sets 
$E,\,F$. A map $\Pi: V\to W$ is a {\sl covering map} for $(G,H)$  if 
\begin{itemize}
\item
$\Pi$ is surjective, 
\item for  each edge $uv\in E$ one has  $\Pi(u)\Pi(v)\in F$
\item  for each edge  $wz\in F$, for each vertex $u\in \Pi^{-1}(w)$   there exists a unique  $v\in V$
  such that 
$uv\in E$  and  $\Pi(v)=z$.
\end{itemize}
\end{definition}
The following is standard and we recall the argument.
\begin{lemma}\label{cover}
Let $\Pi:V\to W$ be a covering map as in Definition \ref{covering}, with $H$ a connected graph, then all fibers of $\Pi$ have the same cardinality 
i.e. $|\Pi^{-1}(w)|=|\Pi^{-1}(w')|$ for all $w,w'\in W$.
\end{lemma}
\begin{proof} 
Let $w,z\in W$  such that  $wz\in F$. For each $u\in \Pi^{-1}(w)$ let ${\mathfrak v}(u)\in\Pi^{-1}(z)$ be the other end  of the edge above $wz$. Similarly for $v\in \Pi^{-1}(z)$ let ${\mathfrak u}(v)\in\Pi^{-1}(w)$ be the other end  of the edge above $wz$, then ${\mathfrak v}:\Pi^{-1}(w)\to\Pi^{-1}(z)$ and 
${\mathfrak u}:\Pi^{-1}(z)\to\Pi^{-1}(w)$ are bijections, inverse to each other. 
It follows that $|\Pi^{-1}(z)|=|\Pi^{-1}(w)|$.  By induction on the distance to $w$ one concludes that the cardinality of the fibers is constant on the connected component of $w$ in $H$.

\end{proof}
\section{The recoil algebra}
\subsection{The product struture}
Let $I,J\subset S$, consider the product space ${\mathcal Y}_I\times {\mathcal Y}_J$ with the product graph structure. 
Let $K\subset S$ and 
\begin{equation}\label{Z}
{\mathcal Z}_{IJK}=\{(\pi,\rho)\in{\mathcal Y}_I\times {\mathcal Y}_J\,|\,\pi\rho\in {\mathcal Y}_K\}
\end{equation}
 with its induced graph structure. It follows from  Proposition \ref{recoil}  that the product map 
\begin{equation}\label{ZY}{\mathcal Z}_{IJK}\to {\mathcal Y}_K:(\pi,\rho)\mapsto \pi\rho
\end{equation} 
is a covering map therefore, by 
Lemma \ref{cover}, for $\sigma\in{\mathcal Y}_K$  the number of pairs $\pi,\rho$ in  ${\mathcal Y}_I\times {\mathcal Y}_J$ such that $\pi\rho=\sigma$
depends only on $I,J,K$. If $a_{IJK}$ is this number then equation (\ref{descentalg}) follows and the
${\mathcal Y}_I$ generate an algebra.
\subsection{A refinement of the coefficients of the descent algebra}
It follows from our result that the graph ${\mathcal Z}_{IJK}$ is a union of connected component, each of which maps to 
${\mathcal Y}_{K}$ with some multiplicity therefore one can write 
\begin{equation}\label{al}
a_{IJK}=\sum_l\lambda^{IJK}_l
\end{equation}
 where $\lambda^{IJK}$ is the integer partition corresponding to the multiplicities of the connected components of ${\mathcal Z}_{IJK}$. This yields a refinement of structure coefficients $a_{IJK}$. A natural problem is to determine which integer partitions can appear in (\ref{al}). In section \ref{sec:lifting}
 we will investigate some further geometric properties of the covering map.
 
\section{Examples}\label{sec:examples}
\subsection{The case of $S_4$}

We consider the permutations with recoil sets $\{1\}$ and $\{3\}$ in $S_4$, their graphs, ordered from left to right in the weak order, are

\bigskip

\begin{center}
\begin{tikzpicture}[scale=1]

\node  at (0,0){$\scriptstyle 2134$};
\node  at (2,0){$\scriptstyle 2314$};
\node  at (4,0){$\scriptstyle 2341$};
\node  at (8,0){$\scriptstyle 1243$};
\node  at (10,0){$\scriptstyle 1423$};
\node  at (12,0){$\scriptstyle 4123$};
\node  at (2,-.7){${\mathcal Y}_1$};
\node  at (10,-.7){${\mathcal Y}_3$};
\draw (.5,0)--(1.5,0);\draw (2.5,0)--(3.5,0);
\draw (8.5,0)--(9.5,0);\draw (10.5,0)--(11.5,0);

\end{tikzpicture}
\end{center}

\bigskip

We draw the recoil classes inside the product graph, labelling the vertices of the graph by the products of the permutations and colouring the horizontal edges  in blue and the vertical ones in red.
The other edges of the product graph are dotted.

\bigskip

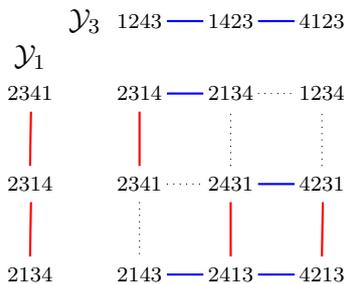
\begin{center}
\begin{tikzpicture}[scale=1.2]

\node  at (-.2,0){$\scriptstyle 2134$};
\node  at (-.2,1){$\scriptstyle 2314$};
\node  at (-.2,2){$\scriptstyle 2341$};
\node  at (1,2.8){$\scriptstyle 1243$};
\node  at (2,2.8){$\scriptstyle 1423$};
\node  at (3,2.8){$\scriptstyle 4123$};

\node  at (1,0){$\scriptstyle 2143$};
\node  at (1,2){$\scriptstyle 2314$};
\node  at (3,0){$\scriptstyle 4213$};
\node  at (1,1){$\scriptstyle 2341$};
\node  at (2,1){$\scriptstyle 2431$};
\node  at (2,0){$\scriptstyle 2413$};
\node  at (2,2){$\scriptstyle 2134$};
\node  at (3,2){$\scriptstyle 1234$};
\node  at (3,1){$\scriptstyle  4231$};
\draw[color=blue,thick] (1.3,0)--(1.7,0);\draw[color=blue,thick] (2.3,0)--(2.7,0);
\draw[color=blue,thick] (2.3,1)--(2.7,1);\draw[color=blue,thick] (1.3,2)--(1.7,2);
\draw[color=blue,thick] (1.3,2.8)--(1.7,2.8);\draw[color=blue,thick] (2.3,2.8)--(2.7,2.8);
\draw[color=red,thick] (1,1.2)--(1,1.8);\draw[color=red,thick] (2,.2)--(2,.8);
\draw[color=red,thick] (3,.2)--(3.01,.8);
\draw[dotted](1,.8)--(1,0.2);\draw[dotted](1.3,1)--(1.7,1);
\draw[dotted](2,1.2)--(2,1.8);\draw[dotted](3,1.2)--(3,1.8);
\draw[dotted](2.3,2)--(2.7,2);
\draw[color=red,thick](-.2,.2)--(-.19,.8);
\draw[color=red,thick](-.2,1.2)--(-.19,1.8);
\node  at (-.2,2.4){${\mathcal Y}_1$};
\node  at (.4,2.8){${\mathcal Y}_3$};
\end{tikzpicture}
\end{center}
We see that $$Y_{1}Y_{3}=Y_\emptyset+Y_{1}+Y_{1,3}.$$

\medskip
Consider now the recoil class ${\mathcal Y}_{2}$, oriented from bottom to top:
$$
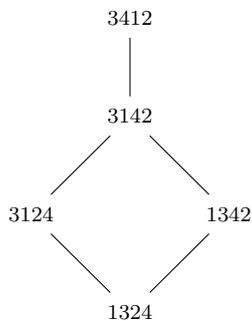
\begin{tikzpicture}[scale=1.3]

\node  at (0,0){$\scriptstyle 1324$};
\node  at (-1,1){$\scriptstyle 3124$};
\node  at (1,1){$\scriptstyle 1342$};
\node  at (0,2){$\scriptstyle 3142$};
\node  at (0,3){$\scriptstyle 3412$};

\draw(-.2,.2)--(-.8,.8);
\draw (.2,.2)--(.8,.8);
\draw (-.8,1.2)--(-.2,1.8);
\draw (.8,1.2)--(.2,1.8);
\draw (0,2.2)--(0,2.8);

\end{tikzpicture}
$$
\medskip

We depict the graph of the product ${\mathcal Y}_{2}{\mathcal Y}_{3}$ below

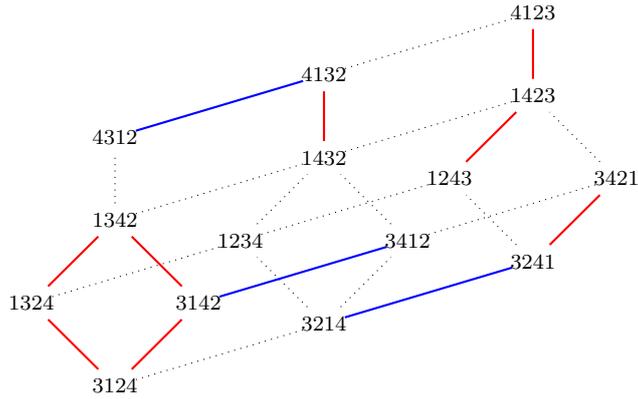
\begin{center}
\begin{tikzpicture}[scale=1.1]

\node  at (0,0){$\scriptstyle 3124$};
\node  at (-1,1){$\scriptstyle 1324$};
\node  at (1,1){$\scriptstyle 3142$};
\node  at (0,2){$\scriptstyle 1342$};
\node  at (0,3){$\scriptstyle 4312$};
\begin{scope}[shift={(2.5,.75)}]
\node  at (0,0){$\scriptstyle 3214$};
\node  at (-1,1){$\scriptstyle 1234$};
\node  at (1,1){$\scriptstyle 3412$};
\node  at (0,2){$\scriptstyle 1432$};
\node  at (0,3){$\scriptstyle 4132$};
\draw[dotted](-.2,.2)--(-.8,.8);
\draw[dotted] (.2,.2)--(.8,.8);
\draw[dotted] (-.8,1.2)--(-.2,1.8);
\draw[dotted] (.8,1.2)--(.2,1.8);
\draw[color=red,thick] (0,2.2)--(0,2.8);
\draw[color=blue,thick] (0.25,0.075)--(2.25,0.675);
\draw[dotted] (-.75,1.075)--(1.25,1.675);
\draw[dotted] (1.25,1.075)--(3.25,1.675);
\draw[dotted] (0.25,2.075)--(2.25,2.675);
\draw[dotted] (0.25,3.075)--(2.25,3.675);
\end{scope}
\begin{scope}[shift={(5,1.5)}]
\node  at (0,0){$\scriptstyle 3241$};
\node  at (-1,1){$\scriptstyle 1243$};
\node  at (1,1){$\scriptstyle 3421$};
\node  at (0,2){$\scriptstyle 1423$};
\node  at (0,3){$\scriptstyle 4123$};
\draw[dotted](-.2,.2)--(-.8,.8);
\draw[color=red,thick] (.2,.2)--(.8,.8);
\draw[color=red,thick] (-.8,1.2)--(-.2,1.8);
\draw[dotted] (.8,1.2)--(.2,1.8);
\draw[color=red,thick] (0,2.2)--(0,2.8);
\end{scope}
\draw[color=red,thick](-.2,.2)--(-.8,.8);
\draw[color=red,thick] (.2,.2)--(.8,.8);
\draw[color=red,thick] (-.8,1.2)--(-.2,1.8);
\draw[color=red,thick] (.8,1.2)--(.2,1.8);
\draw[dotted] (0,2.2)--(0,2.8);

\draw[dotted] (0.25,0.075)--(2.25,0.675);
\draw[dotted] (-.75,1.075)--(1.25,1.675);
\draw[color=blue,thick]  (1.25,1.075)--(3.25,1.675);
\draw[dotted] (0.25,2.075)--(2.25,2.675);
\draw[color=blue,thick] (0.25,3.075)--(2.25,3.675);
\end{tikzpicture}
\end{center}

The picture shows that 
$$Y_{2}Y_{3}=Y_{\emptyset}+Y_{2}+Y_{3}+Y_{2,3}+Y_{1,2}.$$

\subsection{An example in $S_5$}\label{ex_monodromy}

%The class $\{1,3\}$ in $S_5$ is represented below

\begin{center}
\begin{tikzpicture}[scale=.65]

\node at (-8,0)  {Here is the class ${\mathcal Y}_{1,3}$:};
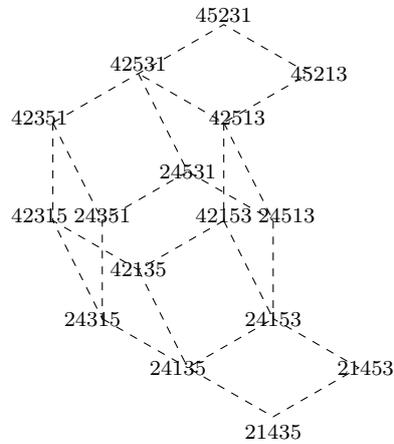
\node  at (0,2.2){$\scriptstyle 42531  $};
\node  at (0,-2){$\scriptstyle 42135 $};
\node  at (2,1.1){$\scriptstyle 42513 $};
\node  at (1.73,-.9){$\scriptstyle 42153 $};
\node  at (-2,1.1){$\scriptstyle 42351  $};
\node  at (-2,-.9){$\scriptstyle 42315 $};
\node  at (1,0){$\scriptstyle  24531 $};
\node  at (.8,-4){$\scriptstyle 24135  $};
\node  at (3,-.9){$\scriptstyle 24513  $};
\node  at (2.73,-3){$\scriptstyle 24153 $};
\node  at (-.73,-.9){$\scriptstyle 24351  $};
\node  at (-.93,-3){$\scriptstyle 24315 $};
\node  at (1.73,3.2){$\scriptstyle 45231 $};
\node  at (3.66,2){$\scriptstyle 45213 $};
\node  at (2.73,-5.3){$\scriptstyle 21435 $};
\node  at (4.6,-4){$\scriptstyle 21453 $};
\draw  [dashed,thin] (-1.73,1)--(0,2)--(1.73,3)--(3.46,2)--(1.73,1)--(0,2);
\draw  [dashed,thin]  (1.73,-1)--(0,-2)--(-1.73,-1)--(-.73,-3)--(-.73,-1)--(1,0);
\draw  [dashed,thin]  (2.73,-1)--(2.73,-3)--(4.46,-4)--(2.73,-5)--(1,-4)--(2.73,-3);
\draw [dashed,thin] (0,2)--(1,0);\draw [dashed] (1.73,1)--(2.73,-1);\draw [dashed] (-1.73,1)--(-.73,-1);
\draw [dashed,thin] (-1.73,1)--(-1.73,-1);\draw [dashed] (1,0)--(2.73,-1);
\draw [dashed,thin] (1.73,1)--(1.73,-1);\draw [dashed] (-.73,-3)--(1,-4);\draw [dashed] (0,-2)--(1,-4);
\draw [dashed,thin] (1.73,-1)--(2.73,-3);
\end{tikzpicture}
\end{center}

It has multiplicity $2$ in the product $Y_{2,3}Y_{3,4}$ whose classes are depicted below
\begin{center}
\begin{tikzpicture}[scale=.45]
\node  at (0,-2.2){$\scriptstyle 41352$};
\node  at (0,2.2){$\scriptstyle 45312$};
\node  at (2,-.9){$\scriptstyle 41532$};
\node  at (2,1.1){$\scriptstyle 45132$};
\node  at (-2,1.1){$\scriptstyle 43512$};
\node  at (-2,-.9){$\scriptstyle 43152$};
\node  at (-3.6,-2){$\scriptstyle 43125$};
\node  at (-2,-3.1){$\scriptstyle 41325$};
\node  at (3.6,-2.1){$\scriptstyle 14532$};
\node  at (2,-3.1){$\scriptstyle 14352$};
\node  at (0,-4.2){$\scriptstyle 14325$};

\node  at (0,-6.2){${\mathcal Y}_{2,3}$};
\draw[dashed,thin] (0,-2)--(1.73,-1)--(1.73,1)--(0,2)--(-1.73,1)--(-1.73,-1)--(0,-2);
\draw[dashed,thin] (-1.73,-1)--(-3.46,-2)--(-1.73,-3)--(0,-2)--(1.73,-3)--(3.46,-2)--(1.73,-1);
\draw[dashed,thin] (-1.73,-3)--(0,-4)--(1.73,-3);

\node  at (10,4){$\scriptstyle 54123$};
\node  at (12,2){$\scriptstyle 51423$};
\node  at (14,0){$\scriptstyle 15423$};
\node  at (10,0){$\scriptstyle 51243$};
\node  at (12,-2){$\scriptstyle 15243$};
\node  at (10,-4){$\scriptstyle 12543$};
\draw[dashed,thin] (10,4)--(14,0)--(10,-4);
\draw[dashed,thin] (12,2)--(10,0)--(12,-2);
\node  at (11,-6.2){${\mathcal Y}_{3,4}$};
\end{tikzpicture}
\end{center}

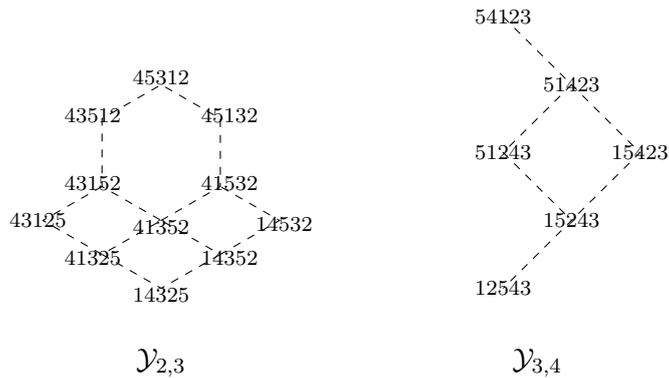
Figure \ref{fig:Z13} shows ${\mathcal Z}_{2,3/3,4/1,3}$ with the vertices indexed by the products in ${\mathcal Y}_{1,3}$. The covering graph is connected.
\begin{figure}
\begin{tikzpicture}[scale=.5]

\node  at (0,2){$\bf\scriptstyle 21453 $};

\node  at (-1.73,1){$\bf\scriptstyle  21435 $};

\draw[dotted] (0,-2)--(1.73,-1)--(1.73,1)--(0,2)--(-1.73,1)--(-1.73,-1)--(0,-2);
\draw[dotted] (-1.73,-1)--(-3.46,-2)--(-1.73,-3)--(0,-2)--(1.73,-3)--(3.46,-2)--(1.73,-1);
\draw[dotted] (-1.73,-3)--(0,-4)--(1.73,-3);
\begin{scope}[shift={(6,-6)}]
%\node  at (0,2){$\color{red}\scriptstyle 41352*51423=24513  $};
\node  at (0,-2){$\bf\scriptstyle 24513  $};
\node  at (0,2){$\bf\scriptstyle 24153 $};
\node  at (1.73,-1){$\bf\scriptstyle  24315 $};
\node  at (1.73,1){$\bf\scriptstyle  24351 $};
\node  at (-1.73,1){$\bf\scriptstyle  24135 $};
\node  at (-1.73,-1){$\bf\scriptstyle  24531 $};

\draw[dotted] (0,-2)--(1.73,-1)--(1.73,1)--(0,2)--(-1.73,1)--(-1.73,-1)--(0,-2);
\draw[dotted] (-1.73,-1)--(-3.46,-2)--(-1.73,-3)--(0,-2)--(1.73,-3)--(3.46,-2)--(1.73,-1);
\draw[dotted] (-1.73,-3)--(0,-4)--(1.73,-3);
\end{scope}
\begin{scope}[shift={(12,-12)}]
\node  at (0,-2){$\bf\scriptstyle 42513  $};
\node  at (0,2){$\bf\scriptstyle 42153 $};
\node  at (1.73,-1){$\bf\scriptstyle  42315 $};
\node  at (1.73,1){$\bf\scriptstyle  42351 $};
\node  at (-1.73,1){$\bf\scriptstyle  42135 $};
\node  at (-1.73,-1){$\bf\scriptstyle  42531 $};
\node  at (-3.46,-2){$\bf\scriptstyle  45231 $};
\node  at (-1.73,-3){$\bf\scriptstyle  45213 $};

\draw[dotted] (0,-2)--(1.73,-1)--(1.73,1)--(0,2)--(-1.73,1)--(-1.73,-1)--(0,-2);
\draw[dotted] (-1.73,-1)--(-3.46,-2)--(-1.73,-3)--(0,-2)--(1.73,-3)--(3.46,-2)--(1.73,-1);
\draw[dotted] (-1.73,-3)--(0,-4)--(1.73,-3);
\end{scope}
\begin{scope}[shift={(0,-12)}]
\node  at (0,-2){$\bf\scriptstyle 24153  $};
\node  at (0,2){$\bf\scriptstyle 24513 $};
\node  at (1.73,-1){$\bf\scriptstyle  24135 $};
\node  at (1.73,1){$\bf\scriptstyle  24531 $};
\node  at (-1.73,1){$\bf\scriptstyle  24315 $};
\node  at (-1.73,-1){$\bf\scriptstyle  24351 $};

\node  at (3.46,-2){$\bf\scriptstyle 21435 $};
\node  at (1.73,-3){$\bf\scriptstyle  21453 $};

\draw[dotted] (0,-2)--(1.73,-1)--(1.73,1)--(0,2)--(-1.73,1)--(-1.73,-1)--(0,-2);
\draw[dotted] (-1.73,-1)--(-3.46,-2)--(-1.73,-3)--(0,-2)--(1.73,-3)--(3.46,-2)--(1.73,-1);
\draw[dotted] (-1.73,-3)--(0,-4)--(1.73,-3);

\end{scope}
\begin{scope}[shift={(6,-18)}]
\node  at (0,-2){$\bf\scriptstyle 42153  $};
\node  at (0,2){$\bf\scriptstyle 42513 $};
\node  at (1.73,-1){$\bf\scriptstyle  42135 $};
\node  at (1.73,1){$\bf\scriptstyle  42531 $};
\node  at (-1.73,1){$\bf\scriptstyle  42315 $};
\node  at (-1.73,-1){$\bf\scriptstyle  42351 $};

\draw[dotted] (0,-2)--(1.73,-1)--(1.73,1)--(0,2)--(-1.73,1)--(-1.73,-1)--(0,-2);
\draw[dotted] (-1.73,-1)--(-3.46,-2)--(-1.73,-3)--(0,-2)--(1.73,-3)--(3.46,-2)--(1.73,-1);
\draw[dotted] (-1.73,-3)--(0,-4)--(1.73,-3);

\end{scope}
\begin{scope}[shift={(0,-24)}]

\node  at (0,2){$\bf\scriptstyle 45213 $};

\node  at (1.73,1){$\bf\scriptstyle  45231 $};

\draw[dotted] (0,-2)--(1.73,-1)--(1.73,1)--(0,2)--(-1.73,1)--(-1.73,-1)--(0,-2);
\draw[dotted] (-1.73,-1)--(-3.46,-2)--(-1.73,-3)--(0,-2)--(1.73,-3)--(3.46,-2)--(1.73,-1);
\draw[dotted] (-1.73,-3)--(0,-4)--(1.73,-3);
\end{scope}

\draw[color=purple,thick] (0,2)--(-1.73,1);
\draw[color=blue] (0,2)--(6,-4);
\draw[color=blue] (-1.73,1)--(4.27,-5);
\draw[color=purple,thick] (6,-8)--(4.27,-7);

\draw[color=blue] (6,-4)--(12,-10);
\draw[color=blue] ((6,-4)--(0,-10);
\draw[color=purple,thick] (6,-4)--(4.27,-5);

\draw[color=blue] (4.27,-5)--(10.27,-11);
\draw[color=blue] (4.27,-5)--(-1.73,-11);

\draw[color=blue] (4.27,-7)--(10.27,-13);
\draw[color=blue] (4.27,-7)--(-1.73,-13);

\draw[color=blue] (7.73,-5)--(13.73,-11);
\draw[color=blue] (7.73,-5)--(1.73,-11);
\draw[color=purple,thick] (7.73,-5)--(7.73,-7);

\draw[color=blue] (7.73,-7)--(13.73,-13);
\draw[color=blue] (7.73,-7)--(1.73,-13);

\draw[color=blue] (6,-8)--(12,-14);
\draw[color=blue] (6,-8)--(0,-14);

\draw[color=purple,thick] (1.73,-11)--(0,-10);
\draw[color=blue] (1.73,-11)--(7.73,-17)--(13.73,-11);

\draw[color=blue] (0,-10)--(6,-16)--(12,-10);

\draw[color=purple,thick] (-1.73,-11)--(-1.73,-13);
\draw[color=blue] (-1.73,-11)--(4.27,-17)--(10.27,-11);

\draw[color=blue] (-1.73,-13)--(4.27,-19)--(10.27,-13);

\draw[color=purple,thick] (0,-14)--(1.73,-13)--(3.46,-14)--(1.73,-15)--(0,-14);
\draw[color=blue] (0,-14)--(6,-20)--(12,-14);
\draw[color=blue] (1.73,-13)--(7.73,-19)--(13.73,-13);

\draw[color=purple,thick] (10.27,-11)--(12,-10);
\draw[color=purple,thick] (10.27,-13)--(12,-14)--(10.27,-15)--(8.54,-14)--(10.27,-13);
\draw[color=purple,thick] (13.73,-11)--(13.73,-13);
\draw[color=purple,thick] (12,-14)--(10.27,-13);
\draw[color=purple,thick] (6,-20)--(7.73,-19);
\draw[color=purple,thick] (6,-16)--(7.73,-17);
\draw[color=purple,thick] (4.27,-19)--(4.27,-17);

\draw[color=blue] (0,-22)--(6,-16);
\draw[color=blue] (1.73,-23)--(7.73,-17);
\draw[color=purple,thick] (0,-22)--(1.73,-23);
\end{tikzpicture}
\caption{${\mathcal Z}_{2,3/3,4/1,3}$}
\label{fig:Z13}
\end{figure}
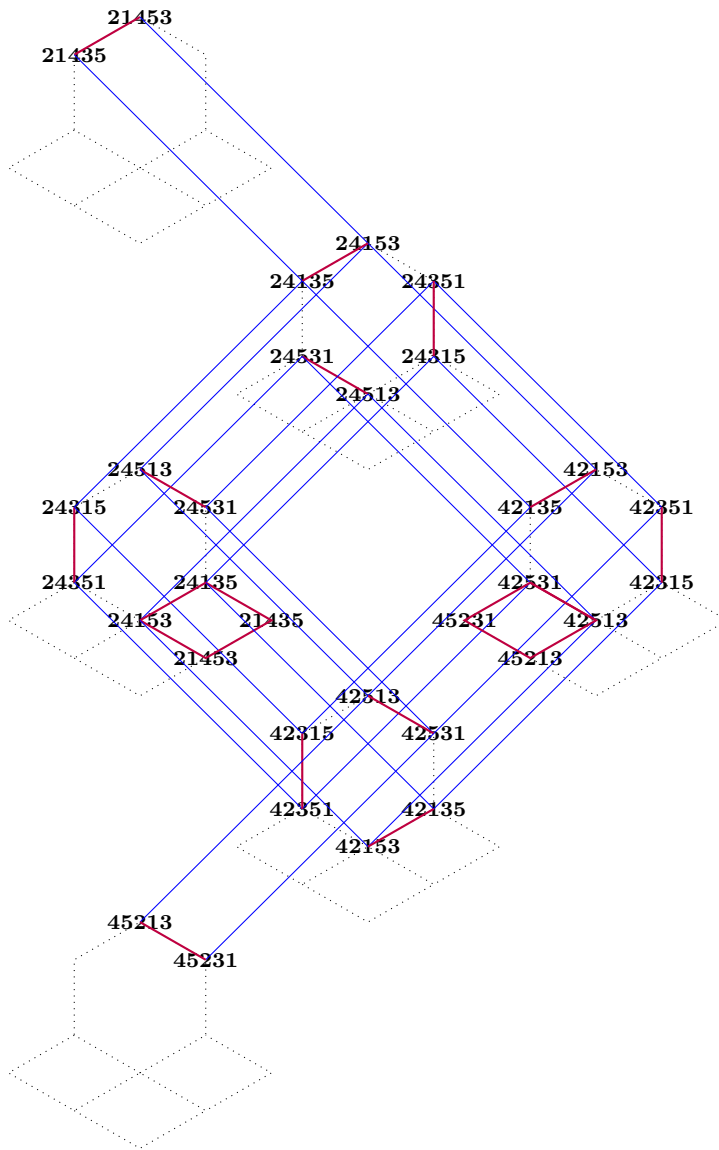

\section{Action of the fundamental group}\label{sec:lifting}
\subsection{Fundamental groups and coverings}
We recall a few basic facts about fundamental groups of graphs, see e.g. \cite{Sp} in the context of algebraic topology. For graphs of course everything can be stated in a purely combinatorial way.
For a finite connected graph $G=(V,E)$ a {\sl path} is a sequence of vertices $v_1-v_2-\ldots-v_l$ such that each $v_iv_{i+1}$ is an edge. A {\sl loop} is a path with $v_1=v_l$. The {\sl fundamental group} $\pi_1(G,v)$, based at some point $v\in V$, is the group of homotopy classes of loops, for the concatenation product. For graphs the homotopy equivalence relation between path is the transitive closure of the relation between two paths of the form
$\ldots -x-\ldots$ and $\ldots -x-y-x-\ldots$ obtained by inserting a loop $x-y-x$ into some path.
The fundamental group  is a free group with $|E|-|V|+1$ generators (in particular it is trivial if $G$ is a tree). For two base points $v,v'$ if $\mu=v-\ldots-v'$ is a path between $v$ and $v'$, with inverse $\mu^{-1}=v'-\ldots -v$, the concatenation map
$\lambda\mapsto \mu^{-1}\lambda\mu$ descends to an isomorphism 
$\pi_1(G,v)\to\pi_1(G,v')$ between the fundamental groups based at $v$ and $v'$.

\subsection{Action of the fundamental group on the fibers of a covering}
If $\Pi:V\to W$ is a covering map from graph  $G$ to $H$ then, for every path $w_1-w_2-\ldots -w_l$ in  $H$ and every $v_1\in\Pi^{-1}(w_1)\subset V$, there exists a unique path $v_1-v_2-\ldots -v_l$ in the graph $G$ which is mapped to the path $w_1-w_2-\ldots- w_l$ by $\Pi$. In particular for any vertex $w\in W$ the fundamental group  $\pi_1(H,w)$, acts by permutations on the fiber
$\Pi^{-1}(w)$: a loop $\lambda$ based at $w$ in $H$ lifts to a path starting at $v\in\Pi^{-1}(w)$ and ending at $[\lambda](v)\in\Pi^{-1}(w)$ so that $[\lambda]$ gives a permutation of the set  $\Pi^{-1}(w)$ and $\lambda\mapsto[\lambda]$ is a homomorphism from $\pi_1(H,w)$ to the group of  permutations of $\Pi^{-1}(w)$. For any two vertices $w,w'\in W$ a path from $w$ to $w'$ gives a bijection between $\Pi^{-1}(w)$ and  $\Pi^{-1}(w')$. One obtains thus a conjugation of the actions  of $\pi_1(H,w)$ and $\pi_1(H,w')$ on their respective fibers $\Pi^{-1}(w)$ and $\Pi^{-1}(w')$.

\subsection{Fundamental groups of recoil classes}
For a recoil class ${\mathcal Y}_K$, with the graph structure induced by the weak order,
the fundamental group is generated by  loops  coming from relations in the symmetric group, (which are  also boundaries of two-dimensional cells in the permutahedron).
These relations are
 
 \begin{eqnarray}
 s_i^2&=&e\label{rel1}\\
  s_is_j&=&s_js_i\qquad |i-j|\geq 2\label{rel2}\\
  s_is_{i+1}s_i&=&s_{i+1}s_is_{i+1}\label{rel3}
  \end{eqnarray}
  
 and the corresponding loops are of the form (provided of course that  the permutations involved belong to  the same recoil class)
$$ \begin{array}{lr}
 \sigma-\sigma s_i- \sigma& \quad relation\ (\ref{rel1})\\
 \sigma-\sigma s_i- \sigma s_is_j-\sigma s_j-\sigma&\quad relation\ (\ref{rel2})\\
 \sigma-\sigma s_i- \sigma s_is_{i+1}-\sigma s_is_{i+1}s_i-\sigma s_{i+1}s_i-\sigma s_{i+1}-\sigma&\quad relation\ (\ref{rel3})\\
 \sigma-\sigma s_{i+1}- \sigma s_{i+1}s_i-\sigma s_{i+1}s_is_{i+1}-\sigma s_is_{i+1}-\sigma s_i-\sigma&\quad relation\ (\ref{rel3})
 \end{array}
 $$
 The last two loops are inverse of each other.
 
 We have considered loops based at some $\sigma$ but of course, as recalled above,  one can change the base point using a path from that point to $\sigma$.

 \subsection{Lifting loops in recoil classes}Let ${\mathcal Y}_I, {\mathcal Y}_J, {\mathcal Y}_K$ be recoil classes with $a_{IJK}\geq 1$ and $\Pi:{\mathcal Z}_{IJK}\to {\mathcal Y}_K$, as in (\ref{Z}), the associated projection. 
 We will determine the possible lifts of generating loops of
 ${\mathcal Y}_K$ in this context and their action on the fibers.
 
 \subsubsection{ $s_i^2=e$}\

 A path $\sigma-\sigma s_i-\sigma s_is_i=\sigma$ in ${\mathcal Y}_K$   is homotopically trivial and acts by the identity on the fibers.
\subsubsection{$s_is_j=s_js_i$}\

  It is easy to see, from the proof of Proposition \ref{recoil}, that any loop in ${\mathcal Y}_K$ of the form 
   $$\sigma-\sigma s_i-\sigma s_is_j-\sigma s_i-\sigma\qquad |i-j|\geq 2$$   lifts to a loop 
   in ${\mathcal Y}_I\times {\mathcal Y}_J$, therefore the action of such a loop is trivial on the fiber.
\subsubsection{ $s_is_{i+1}s_i=s_{i+1}s_is_{i+1}$}\

As we shall see,  a loop associated with such a   relation acts by a permutation   of order 1 or 2 on the fibers.
Let $\sigma\in S_n$ then the vertices in the loop 
\begin{equation}\label{loop}
\sigma-\sigma s_i-\sigma s_is_{i+1}-\sigma s_is_{i+1}s_i-\sigma s_{i+1}s_i-\sigma s_{i+1}-\sigma
\end{equation} belong to the same recoil class as $\sigma$ if and only if  $\sigma_i,\sigma_{i+1},\sigma_{i+2}$ satisfy   
$$|\sigma_i-\sigma_{i+1}|\geq 2,\,|\sigma_{i+1}-\sigma_{i+2}|\geq 2, \,|\sigma_{i}-\sigma_{i+2}|\geq 2.$$  
In the following we assume that $\sigma$ and $i$ satisfy these inequalities.
  Let $\sigma=\pi\rho$ be a factorization with $(\pi,\rho)\in {\mathcal Y}_I\times {\mathcal Y}_J$, there are several cases:
 \begin{itemize}
 \item if $|\rho_i-\rho_{i+1}|\geq 2,\,|\rho_{i+1}-\rho_{i+2}|\geq 2, \,|\rho_{i}-\rho_{i+2}|\geq 2$ then the sequence  
 $$\rho-\rho s_i-\rho s_is_{i+1}-\rho s_is_{i+1}s_i-\rho s_{i+1}s_i-\rho s_{i+1}-\rho$$ is a loop in ${\mathcal Y}_J$ and the loop (\ref{loop}) lifts to the loop 
 $$(\pi,\rho)-(\pi,\rho s_i)-(\pi,\rho s_is_{i+1})-(\pi,\rho s_is_{i+1}s_i)-(\pi,\rho s_{i+1}s_i)-(\pi,\rho s_{i+1})-(\pi,\rho)$$
 in ${\mathcal Z}_{IJK}$. It follows that this loop has a trivial action.
 \item if exactly two of the values $\rho_{i},\rho_{i+1},\rho_{i+2}$ are adjacent then again (\ref{loop}) lifts to a loop in ${\mathcal Z}_{IJK}$. For example one can  check that for $\rho_{i}=j,\rho_{i+1}=j+1,\rho_{i+2}=k$ with $|j-k|,|j+1-k|\geq 2$, the loop
 (\ref{loop}) lifts to 
 $$(\pi,\rho)-(\pi s_j,\rho) -(\pi s_j,\rho s_{i+1})-(\pi s_j,\rho s_{i+1}s_i)-(\pi,\rho s_{i+1}s_i)-(\pi,\rho s_{i+1})-(\pi,\rho)$$ which is a loop in ${\mathcal Z}_{IJK}$ and the action is again trivial. The other cases are similar and left to the reader.
 \item if $(\rho_{i},\rho_{i+1},\rho_{i+2})=(j,j+1,j+2)\ \text{or}\ (j+2,j+1,j)$ then one can check that the loop  (\ref{loop}) lifts either to the loop
 $$(\pi,\rho)-(\pi s_j,\rho) -(\pi s_js_{j+1},\rho )-(\pi s_js_{j+1}s_j,\rho)
  -(\pi s_{j+1}s_j,\rho )-(\pi s_{j+1},\rho )-(\pi,\rho)$$ 
  or to its inverse. Again the action is trivial.
  \item finally if 
  \begin{equation}\label{eq:rhoj}
  (\rho_{i},\rho_{i+1},\rho_{i+2})=(j+1,j,j+2)\ \text{or}\ (j+1,j+2,j)\ \text{or}\ (j,j+2,j+1)\ \text{or}\ (j+2,j,j+1)\end{equation}
   then (\ref{loop}) lifts to a path which is not a loop, for example if $(\rho_{i},\rho_{i+1},\rho_{i+2})=(j+1,j,j+2)$ then, using that $\rho s_{i+1}\rho^{-1}=s_{j}s_{j+1}s_{j}$ one checks that the   lift is
  $$(\pi,\rho)-(\pi s_j,\rho) -(\pi s_j,\rho s_{i+1})-(\pi s_js_{j+1},\rho s_{i+1})
  -( \pi s_js_{j+1} ,\rho )-(\pi s_{j}s_{j+1}s_{j},\rho  )-(\pi s_{j}s_{j+1}s_{j},\rho s_{i+1} )$$ 
  which is not a loop but ends in another factorization of $\sigma$. Starting from this new factorization, the lift of the loop (\ref{loop}) gives now
  $$(\pi s_{j}s_{j+1}s_{j},\rho s_{i+1} )-(\pi s_{j+1}s_{j},\rho s_{i+1}  ) -
 (\pi s_{j+1}s_{j},\rho   )-(\pi s_j,\rho)
  -(\pi s_j,\rho s_{i+1})-(\pi ,\rho s_{i+1})-(\pi,\rho)$$
  which gets back to the original factorization. Thus we see that the action of the loop 
  (\ref{loop}) has order 2. One can check that the  other values in (\ref{eq:rhoj}) also yield an action of order 2.
\end{itemize}
 
Coming back to Figure \ref{fig:Z13} 
we extract from it  an explicit example of this action in 
Figure \ref{fig:braid} which shows the loop projecting on the loop 
$$24153-24135-24315-24351-24531-24513-24153$$
with multiplicity two.

\begin{figure}
\begin{tikzpicture}[scale=.5]

\node  at (0,-2){$\bf\scriptstyle 24513  $};
\node  at (0,2){$\bf\scriptstyle 24153 $};
\node  at (1.73,-1){$\bf\scriptstyle  24315 $};
\node  at (1.73,1){$\bf\scriptstyle  24351 $};
\node  at (-1.73,1){$\bf\scriptstyle  24135 $};
\node  at (-1.73,-1){$\bf\scriptstyle  24531 $};
\draw[dotted] (0,-2)--(1.73,-1)--(1.73,1)--(0,2)--(-1.73,1)--(-1.73,-1)--(0,-2);
\draw[color=purple,thick] (0,2)--(-1.73,1);
\draw[color=purple,thick] (1.73,-1)--(1.73,1);
\draw[color=purple,thick] (-1.73,-1)--(0,-2);
\begin{scope}[shift={(-6,-6)}]
\node  at (0,-2){$\bf\scriptstyle 24153  $};
\node  at (0,2){$\bf\scriptstyle 24513 $};
\node  at (1.73,-1){$\bf\scriptstyle  24135 $};
\node  at (1.73,1){$\bf\scriptstyle  24531 $};
\node  at (-1.73,1){$\bf\scriptstyle  24315 $};
\node  at (-1.73,-1){$\bf\scriptstyle  24351 $};
\draw[dotted] (0,-2)--(1.73,-1)--(1.73,1)--(0,2)--(-1.73,1)--(-1.73,-1)--(0,-2);
\draw[color=purple,thick] (0,2)--(1.73,1);
\draw[color=purple,thick] (-1.73,-1)--(-1.73,1);
\draw[color=purple,thick] (1.73,-1)--(0,-2);
\end{scope}

\draw[color=blue] (0,2)--(-6,-4);
\draw[color=blue] (1.73,1)--(-4.27,-5);
\draw[color=blue] (0,-2)--(-6,-8);
\draw[color=blue] (1.73,-1)--(-4.27,-7);
\draw[color=blue] (-1.73,1)--(-7.73,-5);
\draw[color=blue] (-1.73,-1)--(-7.73,-7);
\end{tikzpicture}
\caption{Lifting a braid relation}
\label{fig:braid}
\end{figure}
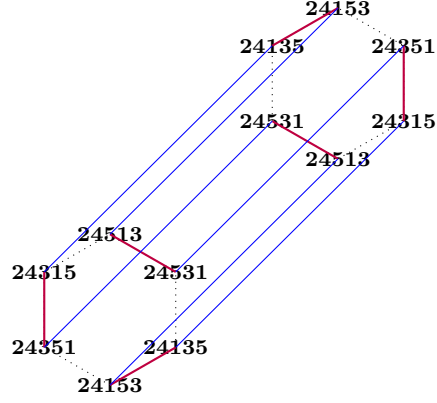

Note that for any $K$ such that the class ${\mathcal Y}_K$ has no loop of type (3) (e.g. the classes
${\mathcal Y}_1,\,{\mathcal Y}_2,\,{\mathcal Y}_3,$ or ${\mathcal Y}_{3,4}$ in the examples of section \ref{sec:examples}) the graph $\Pi^{-1}({\mathcal Y}_K)$ consists of $a_{IJK}$ connected components, all isomorphic to ${\mathcal Y}_K$.
\section{The case of Coxeter groups}
We refer to  \cite{Bou} for basic facts about Coxeter groups.
Let $(W,S)$ be a Coxeter system, where $W$ is a finite Coxeter group and $S$ a system of simple reflections. The length function with respect to $S$ and the right weak order are defined as in the case of symmetric groups. Recall the fundamental {\sl exchange property} (\cite{Bou}  Ch. IV, \textsection 1, Lemme 3).
\begin{lemma}\label{exchange}
Let $w\in W$ and $w=s_1s_2\ldots s_k$ be a reduced decomposition then if $s\in S$ and $\ell (sw)=\ell(w)-1$ there exists $i\in\{1,2,\ldots,k\}$ such that $sw=s_1s_2\ldots s_{i-1}s_{i+1}\ldots s_k$ (this is a reduced decomposition of $sw$).
\end{lemma}

 For $w\in W$ its recoil set is the set of all $s\in S$ such that $\ell(sw)=\ell(w)-1$,   denoted $R(w)$ (again, descents of $w$ are recoils of $w^{-1}$). For $I\subset S$ we denote again by ${\mathcal Y}_I$ the set of $w\in W$ whose recoil set is $I$ and $Y_I=\sum_{w\in {\mathcal Y}_I}w$ is the sum in the group algebra of $W$.
The analog of Proposition \ref{recoilclass} holds unchanged while for Lemma \ref{>2} the anologous result will follow from the following slightly more precise result. 

\begin{lemma} \label{Coxrecoilclass}
Let $w\in W$ and $s\in S$ be such that $\ell(ws)=\ell(w)+1$ then 
\begin{enumerate}
\item
$R(ws)=R(w)$  if  $wsw^{-1}\notin S$ 
\item
 $R(ws)=R(w)\cup\{wsw^{-1}\}\ne R(w)$  if $wsw^{-1}\in S$.
 \end{enumerate}
\end{lemma}
\begin{proof}
 If $w=s_1s_2\ldots s_k$ is a reduced decomposition of $w$ then 
$ws=s_1s_2\ldots s_ks$ is a reduced decomposition of $ws$. 
Suppose that $s'\in R(w)$ then there exists a reduced decomposition of $w$ starting with $s'$ therefore also for $ws$ thus $s'$ is a recoil of $ws$. It follows  that  $R(w)\subset R(ws)$.
Conversely let  $s'$ be a recoil of $ws$ and $ws=s_1s_2\ldots s_ks$ a reduced decomposition, 
by the exchange property, Lemma \ref{exchange},  $s'ws$ has a reduced decomposition obtained by deleting a term in $s_1\ldots s_ks$. If the deleted term is $s_i$ with $1\leq i\leq k$ then $s'w=s_1s_2\ldots s_{i-1}s_{i+1}\ldots s_k$ therefore $s'\in R(w)$.
If the deleted term is the last one,  then $s'ws=w$ and $s'=wsw^{-1}$. In this case $s'\in R(ws)$ but $s'\notin R(w)$ since $s'w=ws$ has length $\ell(w)+1$ thus we have $R(ws)=R(w)\cup\{s'\}\ne R(w)$. This happens if and only if $wsw^{-1}\in S$. 
 The Lemma follows.

\end{proof}
\begin{cor} \label{Cor}
Let $w\in W$ and $s\in S$ then   $R(ws)=R(w)$  if and only if $wsw^{-1}\notin S$.
\end{cor}

Finally the analog of Proposition \ref{recoil} is the following. The proof follows closely the case of the symmetric group, using the above Corollary instead of Lemma \ref{>2} and is left to the reader.

\begin{proposition}\label{Coxrecoil}
Let $w,u,v\in W$  and $s\in S$ be such that $$w=uv\quad \text{and}\quad R(ws)=R(w)$$
 then there exists a unique factorization 
$ws=u'v'$ where either : 
\begin{enumerate}
\item
$u=u'$ while  $v^{-1}v'\in S$ and $R(v)=R(v')$ or
\item
$v=v'$ while $u^{-1}u'\in S$ and $R(u)=R(u')$. 
\end{enumerate}
\end{proposition}

We can now define the sets ${\mathcal Z}_{IJK}$ exactly as in (\ref{Z}) and show that the product maps
$${\mathcal Z}_{IJK}\to {\mathcal Y}_{K}$$ as in (\ref{ZY}), are covering graphs. The rest of the argument to prove the existence of the recoil algebra is the same.

\begin{rem} The proofs of Lemma  \ref{Coxrecoilclass} and  \ref{Coxrecoil} use only the exchange property and not the fact that $W$ is a finite group, therefore they hold also in infinite Coxeter groups. However for infinite Coxeter groups the recoil (or descent) algebra in general is not well defined because of infinite multiplicities: some $a_{IJK}$ may be infinite.
\end{rem}
\subsection{A final example in a dihedral group}
We consider a finite dihedral group with simple generators $s,t$ satisfying the defining relations $s^2=t^2=(st)^n$ for some $n\geq 2$. There are four recoils classes namely ${\mathcal Y}_{\emptyset}=e$, ${\mathcal Y}_{s}=\{s,st,sts,\ldots\}$
${\mathcal Y}_{t}=\{t,ts,tst,\ldots\}$ with $|{\mathcal Y}_{s}|=|{\mathcal Y}_{t}|=n-1$ and ${\mathcal Y}_{s,t}=\{w_0\}$ where $w_0$  is the longest element: 
$$
\begin{array}{rcl}
w_0&=&(st)^{n/2}=(ts)^{n/2}\quad\text{ for $n$ even} \\ w_0&=&t(st)^{(n-1)/2}=s(ts)^{(n-1)/2}\quad\text{ for $n$  odd.}
\end{array}
$$
We show the graph corresponding to the product $Y_{s}Y_{t}$  for $n=6$ so that $w_0=ststst=tststs$.

\medskip

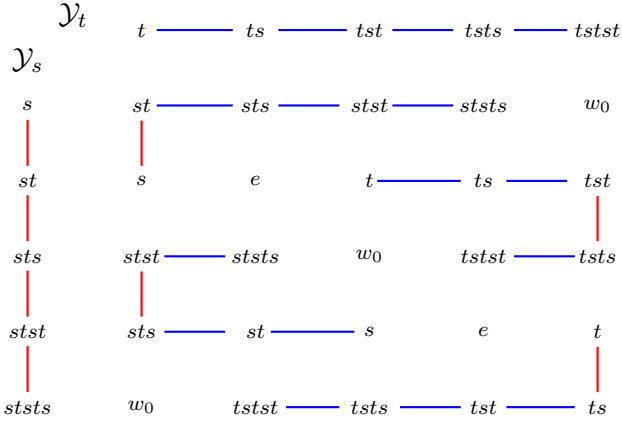
\begin{center}
\begin{tikzpicture}[scale=1]

\node  at (0,0){$\scriptstyle ststs$};
\node  at (0,1){$\scriptstyle stst$};
\node  at (0,2){$\scriptstyle sts$};
\node  at (0,3){$\scriptstyle st$};
\node  at (0,4){$\scriptstyle s$};
\node  at (0,4.6){${\mathcal Y}_s$};

\node  at (1.5,5){$\scriptstyle t$};
\node  at (3,5){$\scriptstyle ts$};
\node  at (4.5,5){$\scriptstyle tst$};
\node  at (6,5){$\scriptstyle tsts$};
\node  at (7.5,5){$\scriptstyle tstst$};
\node  at (0.6,5.2){${\mathcal Y}_t$};

\node  at (1.5,0){$\scriptstyle w_0$};
\node  at (1.5,1){$\scriptstyle sts$};
\node  at (1.5,2){$\scriptstyle stst$};
\node  at (1.5,3){$\scriptstyle s$};
\node  at (1.5,4){$\scriptstyle st$};
\node  at (3,0){$\scriptstyle tstst$};
\node  at (3,1){$\scriptstyle st$};
\node  at (3,2){$\scriptstyle ststs$};
\node  at (3,3){$\scriptstyle e$};
\node  at (3,4){$\scriptstyle sts$};
\node  at (4.5,0){$\scriptstyle tsts$};
\node  at (4.5,1){$\scriptstyle s$};
\node  at (4.5,2){$\scriptstyle w_0$};
\node  at (4.5,3){$\scriptstyle t$};
\node  at (4.5,4){$\scriptstyle stst$};
\node  at (6,0){$\scriptstyle tst$};
\node  at (6,1){$\scriptstyle e$};
\node  at (6,2){$\scriptstyle tstst$};
\node  at (6,3){$\scriptstyle ts$};
\node  at (6,4){$\scriptstyle ststs$};
\node  at (7.5,0){$\scriptstyle ts$};
\node  at (7.5,1){$\scriptstyle t$};
\node  at (7.5,2){$\scriptstyle tsts$};
\node  at (7.5,3){$\scriptstyle tst$};
\node  at (7.5,4){$\scriptstyle w_0$};

\draw[color=red,thick] (1.5,3.2)--(1.5,3.8);\draw[color=red,thick] (7.5,2.2)--(7.5,2.8);
\draw[color=red,thick] (1.5,1.2)--(1.5,1.8);\draw[color=red,thick] (7.5,.2)--(7.5,.8);

\draw[color=red,thick] (0,.2)--(0,.81);
\draw[color=red,thick] (0,1.2)--(0,1.81);
\draw[color=red,thick] (0,2.2)--(0,2.81);
\draw[color=red,thick] (0,3.2)--(0,3.81);

\draw[color=blue,thick] (1.7,4)--(2.7,4);\draw[color=blue,thick] (3.3,4)--(4.1,4);\draw[color=blue,thick] (4.8,4)--(5.6,4);
\draw[color=blue,thick] (4.6,3)--(5.7,3);\draw[color=blue,thick] (6.3,3)--(7.1,3);
\draw[color=blue,thick] (1.8,2)--(2.6,2);\draw[color=blue,thick] (6.4,2)--(7.2,2);
\draw[color=blue,thick] (1.8,1)--(2.6,1);\draw[color=blue,thick] (3.2,1)--(4.3,1);
\draw[color=blue,thick] (3.4,0)--(4.1,0);\draw[color=blue,thick] (4.9,0)--(5.7,0);\draw[color=blue,thick] (6.3,0)--(7.2,0);

\draw[color=blue,thick] (1.7,5)--(2.7,5);
\draw[color=blue,thick] (3.3,5)--(4.2,5);
\draw[color=blue,thick] (4.8,5)--(5.6,5);
\draw[color=blue,thick] (6.4,5)--(7.1,5);
\end{tikzpicture}
\end{center}

One has
$$Y_sY_t=2Y_\emptyset+3Y_{s,t}+2Y_s+2Y_t$$


\begin{thebibliography}{}
\bibitem{Bou}
N. Bourbaki, {\sl
Groupes et alg\`ebres de Lie. Chapitres IV \`a VI.}
Actualit\'es Sci. Indust., No. 1337
Hermann, Paris, 1968.

\bibitem{S}
L. Solomon,  A Mackey formula in the group ring of a Coxeter group,
J. Algebra 41 (1976), no. 2, 255.

\bibitem{Sp}
E. H. 
Spanier, {\sl Algebraic topology},
McGraw-Hill Book Co., New York-Toronto-London, 1966.

\end{thebibliography}
\end{document}